\documentclass{amsproc}
\usepackage{amsmath}
\usepackage{amsfonts}

\theoremstyle{plain}
\newtheorem{acknowledgement}{Acknowledgement}

\newtheorem{corollary}{Corollary}

\newtheorem{proposition}{Proposition}

\newtheorem{theorem}{Theorem}
\numberwithin{equation}{section}
\input{tcilatex}

\begin{document}
\title[A new kind of helicoidal surface]{A new kind of helicoidal surface of
value $m$}
\author{Erhan G\"{u}ler}
\date{}
\dedicatory{In a memory of Professor Franki Dillen (1963-2013)}
\thanks{}

\begin{abstract}
We define a new kind of helicoidal surface of value $m.$ A rotational
surface which is isometric to the helicoidal surface of value $m$ is revealed%
$.$ In addition, we calculate some differential geometric properties of the
helicoidal surface of value $3$ in three dimensional Euclidean space.
\end{abstract}

\maketitle

\section{\textbf{Introduction}}

In classical surface geometry, the right helicoid (resp. catenoid) is the
only ruled (resp. rotational) surface which is minimal in Euclidean space.
If we focus on the ruled (helicoid) and rotational characters, we have
Bour's theorem in \cite{Bo}. The French Mathematician Edmond Bour used
semi-geodesic coordinates and found a number of new cases of the deformation
of surfaces in 1862. He also gave a well known theorem\textit{\ }about the
helicoidal and rotational surfaces.

Kenmotsu, \cite{Ke}\ focuses on the surfaces of revolution with prescribed
mean curvature. About helicoidal surfaces in Euclidean 3-space, do Carmo and
Dajczer \cite{Ca} prove that, by using a result of Bour \cite{Bo}, there
exists a two-parameter family of helicoidal surfaces isometric to a given
helicoidal surface. By making use of this parametrization, they found a
representation formula for helicoidal surfaces with constant mean curvature.
Furthermore they prove that the associated family of Delaunay surfaces is
made up by helicoidal surfaces of constant mean curvature. Hitt and Roussos 
\cite{Hi} also study on the helicoidal surfaces with constant mean curvature
using computer graphics. Baikoussis and Koufogiorgos \cite{Ba} prove that
the helicoidal surfaces satisfying $K_{II}$ $=H$ are locally characterized
by the constancy of the ratio of the principal curvatures. Ikawa determines
pairs of surfaces by Bour's theorem with the additional condition that they
have the same Gauss map in Euclidean 3-space in \cite{I}. Some relations
among the Laplace-Beltrami operator and curvatures of the helicoidal
surfaces in Euclidean 3-space are shown by G\"{u}ler et al in \cite{Gu2}.
They give Bour's theorem on the Gauss map, and some special examples.

On the other hand, Dillen and Sodsiri \cite{Di} study ruled linear
Weingarten surfaces in Minkowski 3-space. See also Minkowskian cases in (%
\cite{Be, Be2, Gu, Gu2, I2}).

In section 2, we recall some basic notions of the Euclidean geometry, and
the reader can be found the definitions of helicoidal and rotational
surfaces of value $m$. Isometric helicoidal and rotational surfaces of value 
$m$ are obtained by Bour's theorem\ in section 3. Finally, isometric
helicoidal and rotational surfaces of value $3$ are examined in the last
section.

\section{\textbf{Preliminaries}}

We shall identify a vector (a,b,c) with its transpose$.$ In this section, we
will obtain the rotational and helicoidal surfaces in Euclidean 3-space. The
reader can be found basic elements of differential geometry in \cite{Bo, Ei,
Sp, St}.

Now we define the rotational surface and helicoidal surface in $\mathbb{E}%
^{3}$. For an open interval $I\subset $ $\mathbb{R}$, let $\gamma :I\overset{%
}{\longrightarrow }\Pi $ be a curve in a plane $\Pi $ in $\mathbb{E}^{3}$,
and let $\ell $ be a straight line in $\Pi $. A rotational surface in $%
\mathbb{E}^{3}$ is defined as a surface rotating a curve $\gamma $ around a
line $\ell $ (these are called the \textit{profile curve} and the \textit{%
axis}, respectively). Suppose that when a profile curve $\gamma $ rotates
around the axis $\ell $, it simultaneously displaces parallel lines
orthogonal to the axis $\ell $, so that the speed of displacement is
proportional to the speed of rotation. Then the resulting surface is called
the \textit{helicoidal surface} with axis $\ell $ and pitch $a\in \mathbb{R}%
^{+}$.

We may suppose that $\ell $ is the line spanned by the vector $(0,0,1)$. The
orthogonal matrix which fixes the above vector is%
\begin{equation*}
A(\theta )=\left( 
\begin{array}{ccc}
\cos \theta & -\sin \theta & 0 \\ 
\sin \theta & \cos \theta & 0 \\ 
0 & 0 & 1%
\end{array}%
\right) ,\text{ \ }\theta \in {\mathbb{R}}.
\end{equation*}%
The matrix $A$ can be found by solving the following equations
simultaneously; $A\ell =\ell ,$ $A^{t}A=AA^{t}=I_{3},$ $\det A=1.$ When the
axis of rotation is $\ell $, there is an Euclidean transformation by which
the axis is $\ell $ transformed to the $z$-axis of $\mathbb{E}^{3}$.
Parametrization of the profile curve is given by $\gamma (r)=(r,0,\varphi
\left( r\right) ),$ where $\varphi \left( r\right) :I\subset {\mathbb{R}}%
\longrightarrow {\mathbb{R}}$ are differentiable function for all $r\in I$.
A helicoidal surface in three dimensional Euclidean space which is spanned
by the vector $(0,0,1)$ with pitch $a$, as follow%
\begin{equation*}
\mathbf{H}(r,\theta )=A(\theta ).\gamma (r)+a\theta \ell .
\end{equation*}%
When $a=0$, helicoidal surface is just a rotational surface.

\section{Helicoidal surfaces of value $m$}

We define a new kind of helicoidal surface, and using Bour's theorem we
reveal a kind of isometric rotational surface in this section.

A \textit{helicoidal surface of value} $m$ is defined by%
\begin{equation}
\mathbf{H}_{m}\left( r,\theta \right) =\mathbf{H}_{m}^{1}\left( r,\theta
\right) +\mathbf{H}_{m}^{2}\left( r,\theta \right) ,  \tag{3.1}
\end{equation}%
where $\mathbf{H}_{m}^{1}\left( r,\theta \right) =\mathbf{\Re }%
_{m}^{1}.\gamma _{m}^{1}+\frac{1}{2}a\theta \ell ,$ $\mathbf{H}%
_{m}^{2}\left( r,\theta \right) =\mathbf{\Re }_{m}^{2}.\gamma _{m}^{2}+\frac{%
1}{2}a\theta \ell ,$ rotating matrices $\mathbf{\Re }_{m}^{1}$ and $\mathbf{%
\Re }_{m}^{2}$ are%
\begin{equation*}
\mathbf{\Re }_{m}^{1}\left( \theta \right) =\left( 
\begin{array}{ccc}
\cos \left[ \left( m-1\right) \theta \right] & \sin \left[ \left( m-1\right)
\theta \right] & 0 \\ 
-\sin \left[ \left( m-1\right) \theta \right] & \cos \left[ \left(
m-1\right) \theta \right] & 0 \\ 
0 & 0 & 1%
\end{array}%
\right)
\end{equation*}%
and%
\begin{equation*}
\mathbf{\Re }_{m}^{2}\left( \theta \right) =\left( 
\begin{array}{ccc}
\cos \left[ \left( m+1\right) \theta \right] & -\sin \left[ \left(
m+1\right) \theta \right] & 0 \\ 
\sin \left[ \left( m+1\right) \theta \right] & \cos \left[ \left( m+1\right)
\theta \right] & 0 \\ 
0 & 0 & 1%
\end{array}%
\right) .
\end{equation*}%
$\ell =\left( 0,0,1\right) $ is the rotating axis, and the profile curves are%
\begin{equation*}
\gamma _{m}^{1}(r)=\left( \frac{r^{m-1}}{m-1},0,\frac{1}{2}\varphi \left(
r\right) \right) ,\text{ \ }\gamma _{m}^{2}(r)=\left( -\frac{r^{m+1}}{m+1},0,%
\frac{1}{2}\varphi \left( r\right) \right) ,
\end{equation*}%
$m\in \mathbb{R}-\left\{ 1\right\} $ in $\gamma _{m}^{1},$ $m\in \mathbb{R}%
-\left\{ -1\right\} $ in $\gamma _{m}^{2},$ $r\in \mathbb{R}^{+},$ $0\leq
\theta \leq 2\pi $, and the pitch $a\in \mathbb{R}^{+}.$ Since the
helicoidal surface is given by rotating the profile curves $\gamma $ around
the axis $\ell $ and simultaneously displacing parallel lines orthogonal to
the axis $\ell $, so that the speed of displacement is proportional to the
speed of rotation. So, we have the following representation of the
helicoidal surface of value $m$ in the following theorem.

\begin{theorem}
A helicoidal surface of value $m$ (in $\left( \text{3.1}\right) $ is reduces
to) 
\begin{equation}
\mathbf{H}_{m}\left( r,\theta \right) =\left( 
\begin{array}{c}
\frac{r^{m-1}}{m-1}\cos \left[ \left( m-1\right) \theta \right] -\frac{%
r^{m+1}}{m+1}\cos \left[ \left( m+1\right) \theta \right] \\ 
-\frac{r^{m-1}}{m-1}\sin \left[ \left( m-1\right) \theta \right] -\frac{%
r^{m+1}}{m+1}\sin \left[ \left( m+1\right) \theta \right] \\ 
\varphi \left( r\right) +a\theta%
\end{array}%
\right) ,  \tag{3.2}
\end{equation}%
is isometric to the rotational surface of value $m$%
\begin{equation}
\mathbf{R}_{m}(r_{\mathbf{R}},\theta _{\mathbf{R}})=\left( 
\begin{array}{c}
\frac{r_{\mathbf{R}}^{m-1}}{m-1}\cos \left[ \left( m-1\right) \theta _{%
\mathbf{R}}\right] -\frac{r_{\mathbf{R}}^{m+1}}{m+1}\cos \left[ \left(
m+1\right) \theta _{\mathbf{R}}\right] \\ 
-\frac{r_{\mathbf{R}}^{m-1}}{m-1}\sin \left[ \left( m-1\right) \theta _{%
\mathbf{R}}\right] -\frac{r_{\mathbf{R}}^{m+1}}{m+1}\sin \left[ \left(
m+1\right) \theta _{\mathbf{R}}\right] \\ 
\varphi _{\mathbf{R}}\left( r_{\mathbf{R}}\right)%
\end{array}%
\right)  \tag{3.3}
\end{equation}%
by Bour's theorem, where%
\begin{eqnarray*}
\varphi _{\mathbf{R}}^{\prime 2} &=&\frac{\left[ 2\left( m+1\right) r_{%
\mathbf{R}}^{2m+1}+4mr_{\mathbf{R}}^{2m-1}\cos \left( 2m\theta _{\mathbf{R}%
}\right) \right] ^{2}\det I}{\left[ 2\left( m+1\right)
r^{2m+1}+4mr^{2m-1}\cos \left( 2m\theta \right) \right] ^{2}G} \\
&&+\frac{2r_{\mathbf{R}}^{2m}\sin ^{2}\left( 2m\theta _{\mathbf{R}}\right) }{%
r_{\mathbf{R}}^{4}+2r_{\mathbf{R}}^{2}\cos \left( 2m\theta _{\mathbf{R}%
}\right) +1} \\
&&-r_{\mathbf{R}}^{2m-4}\left( r_{\mathbf{R}}^{4}-2r_{\mathbf{R}}^{2}\cos
\left( 2m\theta _{\mathbf{R}}\right) +1\right) \text{,}
\end{eqnarray*}%
\begin{eqnarray*}
r_{\mathbf{R}} &=&\sqrt{G}, \\
\theta _{\mathbf{R}} &=&\theta +\int \frac{F}{G}dr, \\
E &=&r^{2m-4}(r^{4}-2r^{2}\cos \left( 2m\theta \right) +1)+\varphi ^{\prime
2}, \\
F &=&2r^{2m-1}\sin \left( 2m\theta \right) +a\varphi ^{\prime }, \\
G &=&r^{2m-2}(r^{4}+2r^{2}\cos \left( 2m\theta \right) +1)+a^{2},
\end{eqnarray*}%
$m\in \mathbb{R}-\left\{ -1,1\right\} ,$ $r\in \mathbb{R}^{+},$ $\theta \in
I\subset \mathbb{R},$ and the pitch $a\in \mathbb{R}^{+}.$
\end{theorem}

\begin{proof}
The line element of the the helicoidal surface $\mathbf{H}_{m}(r,\theta )$ is%
\begin{equation}
\begin{array}{c}
ds^{2}=\left[ r^{2m-4}\left( r^{4}-2r^{2}\cos \left( 2m\theta \right)
+1\right) +\varphi ^{\prime 2}\right] dr^{2} \\ 
+2\left( 2r^{2m-1}\sin \left( 2m\theta \right) +a\varphi ^{\prime }\right)
drd\theta \\ 
\text{ \ \ \ \ \ \ \ \ \ \ \ \ \ }+\left[ r^{2m-2}\left( r^{4}+2r^{2}\cos
\left( 2m\theta \right) +1\right) +a^{2}\right] d\theta ^{2}.%
\end{array}
\tag{3.4}
\end{equation}%
Helices in $\mathbf{H}_{m}(r,\theta )$ are curves defined by $r=const.$. So
curves in $\mathbf{H}_{m}(r,\theta )$ that are orthogonal to helices supply
the orthogonality condition $F$ $dr+G$ $d\theta =0.$ Thus, we obtain $\theta
=-\int \frac{F}{G}$ $dr+c,$ where $c$ is constant. Hence if we put $%
\overline{\theta }=\theta +\int \frac{F}{G}$ $dr,$ then curves orthogonal to
helices are given by $\overline{\theta }=const.$. Substituting the equation $%
d\theta =d\overline{\theta }-\frac{F}{G}dr$ into the line element $\left( 
\text{3.4}\right) $, we have%
\begin{equation}
ds^{2}=\frac{Q}{G}\text{ }dr^{2}+G\text{ }d\overline{\theta }^{2},  \tag{3.5}
\end{equation}%
where $Q:=\det I.$ Setting $\overline{r}:=\dint \sqrt{\frac{Q}{G}}$ $dr,$ $%
k\left( \overline{r}\right) :=\sqrt{G},$ $\left( \text{3.5}\right) $ becomes%
\begin{equation}
ds^{2}=d\overline{r}^{2}+k^{2}\left( \overline{r}\right) d\overline{\theta }%
^{2}.  \tag{3.6}
\end{equation}%
The rotational surface%
\begin{equation}
\mathbf{R}_{m}(r_{\mathbf{R}},\theta _{\mathbf{R}})=\left( 
\begin{array}{c}
\frac{r_{\mathbf{R}}^{m-1}}{m-1}\cos \left[ \left( m-1\right) \theta _{%
\mathbf{R}}\right] -\frac{r_{\mathbf{R}}^{m+1}}{m+1}\cos \left[ \left(
m+1\right) \theta _{\mathbf{R}}\right] \\ 
-\frac{r_{\mathbf{R}}^{m-1}}{m-1}\sin \left[ \left( m-1\right) \theta _{%
\mathbf{R}}\right] -\frac{r_{\mathbf{R}}^{m+1}}{m+1}\sin \left[ \left(
m+1\right) \theta _{\mathbf{R}}\right] \\ 
\varphi _{\mathbf{R}}\left( r_{\mathbf{R}}\right)%
\end{array}%
\right)  \tag{3.7}
\end{equation}%
has the line element%
\begin{equation}
ds_{\mathbf{R}}^{2}=\frac{Q_{\mathbf{R}}}{G_{\mathbf{R}}}\text{ }dr_{\mathbf{%
R}}^{2}+G_{\mathbf{R}}\text{ }d\overline{\theta }_{\mathbf{R}}^{2}, 
\tag{3.8}
\end{equation}%
where%
\begin{eqnarray*}
E_{\mathbf{R}} &=&r_{\mathbf{R}}^{2m-4}\left( r_{\mathbf{R}}^{4}-2r_{\mathbf{%
R}}^{2}\cos \left( 2m\theta _{\mathbf{R}}\right) +1\right) +\varphi _{%
\mathbf{R}}^{\prime 2}, \\
F_{\mathbf{R}} &=&2r_{\mathbf{R}}^{2m-1}\sin \left( 2m\theta _{\mathbf{R}%
}\right) , \\
G_{\mathbf{R}} &=&r_{\mathbf{R}}^{2m-2}(r_{\mathbf{R}}^{4}+2r_{\mathbf{R}%
}^{2}\cos \left( 2m\theta _{\mathbf{R}}\right) +1).
\end{eqnarray*}%
Again, setting $\overline{r}_{\mathbf{R}}:=\dint \sqrt{\frac{Q_{\mathbf{R}}}{%
G_{\mathbf{R}}}\text{ }}dr_{\mathbf{R}},$ $\ k_{\mathbf{R}}\left( \overline{r%
}_{\mathbf{R}}\right) :=\sqrt{G_{\mathbf{R}}},$ then $\left( \text{3.8}%
\right) $ becomes%
\begin{equation}
ds_{\mathbf{R}}^{2}=d\overline{r}_{\mathbf{R}}^{2}+k_{\mathbf{R}}^{2}\left( 
\overline{r}_{\mathbf{R}}\right) d\overline{\theta }_{\mathbf{R}}^{2}. 
\tag{3.9}
\end{equation}%
Comparing $\left( \text{3.6}\right) $ with $\left( \text{3.9}\right) $, if
we take $\overline{r}=\overline{r}_{\mathbf{R}},$ \ $\overline{\theta }=%
\overline{\theta }_{\mathbf{R}},$ \ $k\left( \overline{r}\right) =k_{\mathbf{%
R}}\left( \overline{r}_{\mathbf{R}}\right) ,$ then we have an isometry
between $\mathbf{H}_{m}(r,\theta )$ and $\mathbf{R}_{m}(r_{\mathbf{R}%
},\theta _{\mathbf{R}})$. Therefore, it follows that%
\begin{equation}
\dint \sqrt{\frac{Q}{G}}\text{ }dr=\dint \sqrt{\frac{Q_{\mathbf{R}}}{G_{%
\mathbf{R}}}}\text{ }dr_{\mathbf{R}}.  \tag{3.10}
\end{equation}%
Substituting the equation%
\begin{equation*}
dr_{\mathbf{R}}=\frac{2\left( m+1\right) r^{2m+1}+4mr^{2m-1}\cos \left(
2m\theta \right) }{2\left( m+1\right) r_{\mathbf{R}}^{2m+1}+4mr_{\mathbf{R}%
}^{2m-1}\cos \left( 2m\theta _{\mathbf{R}}\right) }\text{ }dr
\end{equation*}%
into the $\left( \text{3.10}\right) $, we get the function $\varphi _{%
\mathbf{R}}.$
\end{proof}

\section{Helicoidal surface of value $3$}

We give the helicoidal surface of value $3$ using Bour's theorem in this
section.

\begin{proposition}
A helicoidal surface of value $3$ (see Figure 1 a-b)%
\begin{equation}
\mathbf{H}_{3}\left( r,\theta \right) =\left( 
\begin{array}{c}
\frac{r^{2}}{2}\cos \left( 2\theta \right) -\frac{r^{4}}{4}\cos \left(
4\theta \right) \\ 
-\frac{r^{2}}{2}\sin \left( 2\theta \right) -\frac{r^{4}}{4}\sin \left(
4\theta \right) \\ 
\varphi \left( r\right) +a\theta%
\end{array}%
\right) ,  \tag{3.11}
\end{equation}%
is isometric to the rotational surface of value $3$%
\begin{equation}
\mathbf{R}_{3}(r_{\mathbf{R}},\theta _{\mathbf{R}})=\left( 
\begin{array}{c}
\frac{G}{2}\cos \left[ 2\left( \theta +\int \frac{F}{G}dr\right) \right] -%
\frac{G^{2}}{4}\cos \left[ 4\left( \theta +\int \frac{F}{G}dr\right) \right]
\\ 
-\frac{G}{2}\sin \left[ 2\left( \theta +\int \frac{F}{G}dr\right) \right] -%
\frac{G^{2}}{4}\sin \left[ 4\left( \theta +\int \frac{F}{G}dr\right) \right]
\\ 
\varphi _{\mathbf{R}}\left( r_{\mathbf{R}}\right)%
\end{array}%
\right) ,  \tag{3.12}
\end{equation}%
where%
\begin{eqnarray*}
\varphi _{\mathbf{R}}^{\prime 2} &=&\frac{\left\{ 8G^{7/2}+12G^{5/2}\cos %
\left[ 6\left( \theta +\int \frac{F}{G}dr\right) \right] \right\} ^{2}\det I%
}{\left[ 8r^{7}+12r^{5}\cos \left( 6\theta \right) \right] ^{2}G} \\
&&+\frac{2G^{3}\sin ^{2}\left[ 6\left( \theta +\int \frac{F}{G}dr\right) %
\right] }{G^{2}+2G\cos \left[ 6\left( \theta +\int \frac{F}{G}dr\right) %
\right] +1} \\
&&-G\left\{ G^{2}-2G\cos \left[ 6\left( \theta +\int \frac{F}{G}dr\right) %
\right] +1\right\} \text{,}
\end{eqnarray*}%
\begin{eqnarray*}
E &=&r^{2}[r^{4}-2r^{2}\cos \left( 6\theta \right) +1]+\varphi ^{\prime 2},
\\
F &=&2r^{5}\sin \left( 6\theta \right) +a\varphi ^{\prime }, \\
G &=&r^{4}[r^{4}+2r^{2}\cos \left( 6\theta \right) +1]+a^{2},
\end{eqnarray*}%
$\det I=EG-F^{2},$ $r,a\in \mathbb{R}^{+},$ $0\leq \theta \leq 2\pi .$
\end{proposition}

\begin{proof}
Taking $m=3$ in the previous theorem, we easily get the results.
\end{proof}

\begin{corollary}
When $a=0$ and $\varphi \left( r\right) =\frac{2}{3}r^{3}\cos \left( 3\theta
\right) $ in $\mathbf{H}_{3}\left( r,\theta \right) ,$ we obtain the Bour's
minimal surface $\mathfrak{B}_{3}\left( r,\theta \right) $ (see Figure 2
a-b, and \cite{Gu3} for details).
\end{corollary}

\begin{proposition}
The mean curvature and the Gaussian curvature of the helicoidal surface of
value $3$ are as follow%
\begin{eqnarray*}
H &=&\frac{1}{4\left( \det I\right) ^{3/2}}\{2r^{3}(r^{4}-1)(r^{8}+2r^{6}%
\cos (6\theta )+r^{4}+a^{2})\varphi ^{\prime \prime } \\
&&+4r^{4}(2r^{4}+r^{2}\cos 6\theta -1)\varphi ^{\prime 3}-12ar^{5}\sin
(6\theta )\varphi ^{\prime 2} \\
&&+r^{2}[2(r^{2}+1)(r^{8}+r^{4}+a^{2})\cos (2\theta
)+(2r^{2}-1)(r^{8}+r^{4}+a^{2})\cos (4\theta ) \\
&&-2r^{2}(10r^{8}+6r^{6}\cos (6\theta )-6r^{4}+5a^{2})\allowbreak \cos
(6\theta ) \\
&&+3r^{4}(r^{8}+r^{4}+a^{2})\cos (8\theta )-16r^{8}\sin ^{2}6\theta
\allowbreak \\
&&+\allowbreak 4r^{6}(r^{2}+1)\cos (2\theta )\cos (6\theta
)+2r^{6}(2r^{2}-1)\cos (4\theta )\cos (6\theta ) \\
&&+6r^{10}\cos 6\theta \cos 8\theta +\allowbreak r^{4}\left(
13a^{2}+2r^{4}+5r^{8}-3\right) -7a^{2}]\varphi ^{\prime }\allowbreak \\
&&+\allowbreak 2(2ar^{3}+1)(r^{8}+r^{4}+a^{2})\sin (2\theta )+\allowbreak
r(a+2r)(r^{8}+r^{4}+a^{2})\sin (4\theta ) \\
&&+\allowbreak 2ar^{3}(15r^{8}-9r^{4}+a^{2})\allowbreak \sin (6\theta
)-3ar^{5}(r^{8}+r^{4}+a^{2})\sin (8\theta ) \\
&&+4r^{6}(2ar^{3}+1)\sin (2\theta )\cos (6\theta )+2r^{7}(a+2r)\sin (4\theta
)\cos (6\theta ) \\
&&-2ar^{9}(2\sin (6\theta )+3r^{2}\sin (8\theta ))\allowbreak \cos (6\theta
)\}
\end{eqnarray*}%
and%
\begin{eqnarray*}
K &=&\frac{1}{\left( \det I\right) ^{2}}\{2ar^{5}(r^{4}-1)\allowbreak \left(
r^{2}\cos (6\theta )-2r^{4}+1\right) \varphi ^{\prime \prime
}+2r^{8}(r^{4}-1)\sin (6\theta )\varphi ^{\prime }\varphi ^{\prime \prime }
\\
&&+\frac{1}{2}r^{7}[-32r^{9}+28r^{5}-8r+(-3r^{4}+2r^{2}-1)\sin (2\theta
)+2(r^{2}+1)\sin (4\theta ) \\
&&+2(-3r^{4}+1)\sin (6\theta )+2(r^{2}+1)\sin (8\theta )\allowbreak
+(2r^{2}-1)\sin (10\theta )\allowbreak -2r^{2}\sin (12\theta ) \\
&&+3r^{4}\sin (14\theta )+16r^{3}(-2r^{4}+1)\cos (6\theta )-4r^{5}\cos
(12\theta )]\varphi ^{\prime 2} \\
&&+r^{4}[8ar^{5}(2r^{4}-1)\sin (6\theta )+4ar^{7}\sin (12\theta
)+(-4ar^{6}-3ar^{4}+r^{3}+2ar^{2}+2a)\cos (2\theta ) \\
&&+(-4ar^{6}+5ar^{4}+3ar^{2}+r-a)\cos (4\theta )+a(r^{4}-1)r^{2}\cos
(6\theta ) \\
&&+(-6ar^{8}+2ar^{4}+ar^{2}-r)\cos (8\theta )+(ar^{4}-r^{3}-ar^{2})\cos
(10\theta ) \\
&&-2ar^{4}\cos (12\theta )+3ar^{6}\cos (14\theta )+6ar^{8}-5ar^{4}+a]\varphi
^{\prime }\allowbreak \\
&&-\frac{1}{2}ar^{2}\left( 19ar^{7}-7ar^{3}+10r^{2}-4\right) \allowbreak
\sin (2\theta )+ar^{2}\left( 1-4r^{4}\right) \left( ar+1\right) \allowbreak
\sin (4\theta ) \\
&&-2a^{2}r^{5}\left( 2r^{4}-1\right) \sin (6\theta )+ar^{4}\left(
6ar^{7}-ar^{3}+1\right) \sin (8\theta )\allowbreak +\frac{1}{2}ar^{4}\left(
ar+2\right) \sin (10\theta ) \\
&&+\allowbreak a^{2}r^{7}\sin (12\theta )-\frac{3}{2}a^{2}r^{9}\sin
(14\theta )\allowbreak +2a^{2}r^{10}\cos (12\theta )-2a^{2}r^{10}\}.
\end{eqnarray*}%
respectively, where $\det I=EG-F^{2},$ $\varphi ^{\prime }=\frac{d\varphi }{%
dr},$ $r,a\in \mathbb{R}^{+},$ $0\leq \theta \leq 2\pi $.
\end{proposition}

\begin{proof}
Taking the differential with respect to $r,$ $\theta $ to the $\mathbf{H}%
_{3},$ we have%
\begin{equation*}
\left( \mathbf{H}_{3}\right) _{r}=\left( 
\begin{array}{c}
r\cos \left( 2\theta \right) -r^{3}\cos \left( 4\theta \right) \\ 
-r\sin \left( 2\theta \right) -r^{3}\sin \left( 4\theta \right) \\ 
\varphi ^{\prime }%
\end{array}%
\right) .
\end{equation*}%
and%
\begin{equation*}
\left( \mathbf{H}_{3}\right) _{\theta }=\left( 
\begin{array}{c}
-r^{2}\sin (2\theta )+r^{4}\sin (4\theta ) \\ 
-r^{2}\cos (2\theta )-r^{4}\cos (4\theta ) \\ 
a%
\end{array}%
\right) ,
\end{equation*}%
The coefficients of the first fundamental form of the surface are%
\begin{equation*}
E=r^{2}(r^{4}-2r^{2}\cos (6\theta )+1)+\varphi ^{\prime 2},
\end{equation*}%
\begin{equation*}
F=2r^{5}\sin \left( 6\theta \right) +a\varphi ^{\prime },
\end{equation*}%
\begin{equation*}
G=r^{4}\left( r^{4}+2r^{2}\cos \left( 6\theta \right) +1\right) +a^{2}.
\end{equation*}%
Then we get%
\begin{eqnarray*}
\det I &=&r^{2}[r^{12}-2r^{8}+a^{2}r^{4}+r^{4}-2a^{2}r^{2}\cos \left(
6\theta \right) +a^{2}] \\
&&-4ar^{5}\sin \left( 6\theta \right) \varphi ^{\prime }+r^{4}\left[
r^{4}+2r^{2}\cos \left( 6\theta \right) +1\right] \varphi ^{\prime 2}.
\end{eqnarray*}%
Using the second differentials 
\begin{equation*}
\left( \mathbf{H}_{3}\right) _{rr}=\left( 
\begin{array}{c}
\cos \left( 2\theta \right) -3r^{2}\cos \left( 4\theta \right) \\ 
-\sin \left( 2\theta \right) -3r^{2}\sin \left( 4\theta \right) \\ 
\varphi ^{\prime \prime }%
\end{array}%
\right) ,
\end{equation*}%
\begin{equation*}
\left( \mathbf{H}_{3}\right) _{r\theta }=\left( 
\begin{array}{c}
-2r\sin (2\theta )+4r^{3}\sin (4\theta ) \\ 
-2r\cos (2\theta )-4r^{3}\cos (4\theta ) \\ 
0%
\end{array}%
\right) ,
\end{equation*}%
\begin{equation*}
\left( \mathbf{H}_{3}\right) _{\theta \theta }=\left( 
\begin{array}{c}
-2r^{2}\cos (2\theta )+4r^{4}\cos (4\theta ) \\ 
2r^{2}\sin (2\theta )+4r^{4}\sin (4\theta ) \\ 
0%
\end{array}%
\right) ,
\end{equation*}%
and the Gauss map (the unit normal) 
\begin{equation*}
e=\frac{1}{\sqrt{\det I}}\left( 
\begin{array}{c}
-ar(\allowbreak r^{2}\sin \left( 4\theta \right) +\sin \left( 2\theta
\right) )+r^{2}\varphi ^{\prime }(r^{2}\cos \left( 4\theta \right) +\cos
\left( 2\theta \right) ) \\ 
ar(r^{2}\cos \left( 4\theta \right) -\cos \left( 2\theta \right)
)+r^{2}\varphi ^{\prime }(r^{2}\sin \left( 4\theta \right) -\sin \left(
2\theta \right) ) \\ 
r^{7}-r^{3}%
\end{array}%
\right)
\end{equation*}%
of the surface $\mathbf{H}_{3}$, we have the coefficients of the second
fundamental form of the surface as follow%
\begin{eqnarray*}
L &=&\frac{1}{\sqrt{\det I}}(r^{3}(r^{4}-1)\varphi ^{\prime \prime }+\frac{1%
}{2}r^{2}(1-3r^{4}+2(1+r^{2})\cos 2\theta \\
&&+(-1+2r^{2})\cos 4\theta -2r^{2}\cos 6\theta +3r^{4}\cos 8\theta )\varphi
^{\prime } \\
&&+(1+2ar^{3})\sin 2\theta +\frac{1}{2}r(a+2r)\sin 4\theta \\
&&+ar^{3}\sin 6\theta -\frac{3}{2}ar^{5}\sin 8\theta ),
\end{eqnarray*}%
\begin{equation*}
M=\frac{1}{\sqrt{\det I}}2r^{2}((-2r^{4}+r^{2}\cos 6\theta +1)a+r^{3}\varphi
^{\prime }\sin 6\theta \allowbreak ),
\end{equation*}%
and%
\begin{equation*}
N=\frac{1}{\sqrt{\det I}}2r^{4}(-ar\sin 6\theta +(2r^{4}+r^{2}\cos 6\theta
-1)\varphi ^{\prime }).
\end{equation*}%
\ Therefore, we can see the results easily.
\end{proof}

\begin{corollary}
If the helicoidal surface of value $3$ is minimal then we get%
\begin{eqnarray*}
0 &=&2r^{3}(r^{4}-1)(r^{8}+2r^{6}\cos (6\theta )+r^{4}+a^{2})\varphi
^{\prime \prime } \\
&&+4r^{4}(2r^{4}+r^{2}\cos 6\theta -1)\varphi ^{\prime 3}-12ar^{5}\sin
(6\theta )\varphi ^{\prime 2} \\
&&+r^{2}[2(r^{2}+1)(r^{8}+r^{4}+a^{2})\cos (2\theta
)+(2r^{2}-1)(r^{8}+r^{4}+a^{2})\cos (4\theta ) \\
&&-2r^{2}(10r^{8}+6r^{6}\cos (6\theta )-6r^{4}+5a^{2})\allowbreak \cos
(6\theta ) \\
&&+3r^{4}(r^{8}+r^{4}+a^{2})\cos (8\theta )-16r^{8}\sin ^{2}6\theta
\allowbreak \\
&&+\allowbreak 4r^{6}(r^{2}+1)\cos (2\theta )\cos (6\theta
)+2r^{6}(2r^{2}-1)\cos (4\theta )\cos (6\theta ) \\
&&+6r^{10}\cos 6\theta \cos 8\theta +\allowbreak r^{4}\left(
13a^{2}+2r^{4}+5r^{8}-3\right) -7a^{2}]\varphi ^{\prime }\allowbreak \\
&&+\allowbreak 2(2ar^{3}+1)(r^{8}+r^{4}+a^{2})\sin (2\theta )+\allowbreak
r(a+2r)(r^{8}+r^{4}+a^{2})\sin (4\theta ) \\
&&+\allowbreak 2ar^{3}(15r^{8}-9r^{4}+a^{2})\allowbreak \sin (6\theta
)-3ar^{5}(r^{8}+r^{4}+a^{2})\sin (8\theta ) \\
&&+4r^{6}(2ar^{3}+1)\sin (2\theta )\cos (6\theta )+2r^{7}(a+2r)\sin (4\theta
)\cos (6\theta ) \\
&&-2ar^{9}(2\sin (6\theta )+3r^{2}\sin (8\theta ))\allowbreak \cos (6\theta )
\end{eqnarray*}%
The solution of the second order ODE as above is an atractive
problem.\bigskip
\end{corollary}

\begin{equation*}
\FRAME{itbpF}{2.2355in}{2.4353in}{0in}{}{}{Figure}{\special{language
"Scientific Word";type "GRAPHIC";maintain-aspect-ratio TRUE;display
"USEDEF";valid_file "T";width 2.2355in;height 2.4353in;depth
0in;original-width 3.6979in;original-height 4.0309in;cropleft "0";croptop
"1";cropright "1";cropbottom "0";tempfilename
'N105O502.wmf';tempfile-properties "XPR";}}\FRAME{itbpF}{2.1837in}{2.4483in}{%
0in}{}{}{Figure}{\special{language "Scientific Word";type
"GRAPHIC";maintain-aspect-ratio TRUE;display "USEDEF";valid_file "T";width
2.1837in;height 2.4483in;depth 0in;original-width 3.9375in;original-height
4.4166in;cropleft "0";croptop "1";cropright "1";cropbottom "0";tempfilename
'N0ZF9W01.wmf';tempfile-properties "XPR";}}
\end{equation*}

\begin{center}
(a) \ \ \ \ \ \ \ \ \ \ \ \ \ \ \ \ \ \ \ \ \ \ \ \ \ \ \ \ \ \ \ \ \ \ \ \
\ \ \ \ \ \ \ \ \ \ \ \ \ \ \ \ (b)\bigskip

Figure 1. $\ $Helicoidal surface of value $3,$ $\varphi \left( r\right) =%
\frac{2}{3}r^{3}\cos \left( 3\theta \right) $

\bigskip

\begin{equation*}
\FRAME{itbpF}{2.2352in}{2.339in}{0in}{}{}{Figure}{\special{language
"Scientific Word";type "GRAPHIC";maintain-aspect-ratio TRUE;display
"USEDEF";valid_file "T";width 2.2352in;height 2.339in;depth
0in;original-width 3.8992in;original-height 4.081in;cropleft "0";croptop
"1";cropright "1";cropbottom "0";tempfilename
'N13JL000.wmf';tempfile-properties "XPR";}}\FRAME{itbpF}{2.4019in}{2.4339in}{%
0in}{}{}{Figure}{\special{language "Scientific Word";type
"GRAPHIC";maintain-aspect-ratio TRUE;display "USEDEF";valid_file "T";width
2.4019in;height 2.4339in;depth 0in;original-width 4.0278in;original-height
4.081in;cropleft "0";croptop "1";cropright "1";cropbottom "0";tempfilename
'N13JOH01.wmf';tempfile-properties "XPR";}}
\end{equation*}%
(a) \ \ \ \ \ \ \ \ \ \ \ \ \ \ \ \ \ \ \ \ \ \ \ \ \ \ \ \ \ \ \ \ \ \ \ \
\ \ \ \ \ \ \ \ \ \ \ \ \ \ \ \ (b)\bigskip

Figure 2. $\ $Bour's minimal surface of value $3,$ $a=0,$ $\varphi \left(
r\right) =\frac{2}{3}r^{3}\cos \left( 3\theta \right) $

\bigskip
\end{center}

The author also focuses on the spacelike and timelike helicoidal surfaces of
value $m$ in the Minkowski 3-space $\mathbb{L}^{3}$ in the next papers$.$

\begin{acknowledgement}
This work had been stated by the author, when he visited as a post-doctoral
researcher at the Katholieke Universiteit Leuven, Belgium in 2011-2012
academic year. The author would like to thank to the hospitality of the
members of the geometry section at K.U. Leuven. Especially to the Professor
Franki Dillen, and Dr. Ana Nistor.
\end{acknowledgement}

Erhan G\"{u}ler\newline
Bart\i n University, Faculty of Science, Department of Mathematics, 74100
Bart\i n, Turkey\newline
ergler@gmail.com\newline
\hfill \endgroup

\end{document}